\documentclass[a4paper,notitlepage,10pt]{article}%
\usepackage{amssymb}
\usepackage{graphicx}
\usepackage{amsmath}
\usepackage{amsthm}
\usepackage{dsfont}
\usepackage{amsfonts}%
\providecommand{\U}[1]{\protect\rule{.1in}{.1in}}
\theoremstyle{plain}
\newtheorem{theorem}{Theorem}[section]

\newtheorem{corollary}[theorem]{Corollary}

\newtheorem{lemma}[theorem]{Lemma}

\newtheorem{proposition}[theorem]{Proposition}

\theoremstyle{definition}

\newtheorem{remark}[theorem]{Remark}
\numberwithin{equation}{section}
\numberwithin{theorem}{section}
\ifx\pdfoutput\relax\let\pdfoutput=\undefined\fi
\newcount\msipdfoutput
\ifx\pdfoutput\undefined\else
\ifcase\pdfoutput\else
\ifx\paperwidth\undefined\else
\ifdim\paperheight=0pt\relax\else\pdfpageheight\paperheight\fi
\ifdim\paperwidth=0pt\relax\else\pdfpagewidth\paperwidth\fi
\fi\fi\fi
\begin{document}

\title{Uniqueness and sign properties of minimizers in a quasilinear indefinite
problem \thanks{2020 \textit{Mathematics Subject Classification}. 35J25,
35J62, 35J92.} \thanks{\textit{Key words and phrases}. quasilinear,
indefinite, sublinear, uniqueness.} }
\author{Uriel Kaufmann\thanks{FaMAF-CIEM (CONICET), Universidad Nacional de
C\'{o}rdoba, Medina Allende s/n, Ciudad Universitaria, 5000 C\'{o}rdoba,
Argentina. \textit{E-mail address: }kaufmann@mate.uncor.edu} , Humberto Ramos
Quoirin \thanks{CIEM-FaMAF, Universidad Nacional de C\'{o}rdoba, (5000)
C\'{o}rdoba, Argentina. \textit{E-mail address: }humbertorq@gmail.com} ,
Kenichiro Umezu\thanks{Department of Mathematics, Faculty of Education,
Ibaraki University, Mito 310-8512, Japan. \textit{E-mail address:
}kenichiro.umezu.math@vc.ibaraki.ac.jp}
\and \noindent}
\maketitle

\begin{abstract}
Let $1<q<p$ and $a\in C(\overline{\Omega})$ be sign-changing, where $\Omega$
is a bounded and smooth domain of $\mathbb{R}^{N}$. We show that the
functional
\[
I_{q}(u):=\int_{\Omega}\left(  \frac{1}{p}|\nabla u|^{p}-\frac{1}%
{q}a(x)|u|^{q}\right)  ,
\]
has exactly one nonnegative minimizer $U_{q}$ (in $W_{0}^{1,p}(\Omega)$ or
$W^{1,p}(\Omega)$). In addition, we prove that $U_{q}$ is the only possible
\textit{positive} solution of the associated Euler-Lagrange equation, which
shows that this equation has at most one positive
solution. Furthermore, we show that if $q$ is close enough to $p$ then $U_{q}$
is positive, which also guarantees that minimizers of $I_{q}$ do not change
sign. Several of these results are new even for $p=2$.

\end{abstract}


\section{Introduction}

Let $\Omega$ be a bounded and smooth domain of $\mathbb{R}^{N}$ with $N\geq1$.
This note is concerned with the problem
\[
\left\{
\begin{array}
[c]{lll}%
-\Delta_{p}u=a(x)u^{q-1} & \mathrm{in} & \Omega,\\
u\geq0 & \mathrm{in} & \Omega,\\
\mathbf{B}u=0 & \mathrm{on} & \partial\Omega,
\end{array}
\right.  \leqno{(P_q)}
\]
where $\Delta_{p}$ is the $p$-Laplacian operator. Here $a\in C(\overline
{\Omega})$ changes sign and $q\in(1,p)$ (which is known as the $p$-sublinear
or $p$-subhomogeneous case).

We consider either Dirichlet ($\mathbf{B }u=u$) or Neumann ($\mathbf{B
}u=\partial_{\nu}u$, where $\nu$ is the outward unit normal to $\partial
\Omega$) homogeneous boundary conditions. In the Neumann case, we assume
throughout this note that $\int_{\Omega}a<0$, which is a necessary condition
for the existence of a positive solution of $(P_{q})$.

By a \textit{solution} of $(P_{q})$ we mean a nonnegative weak solution, i.e.
$u\in X$ such that $u\geq0$ in $\Omega$ and
\[
\int_{\Omega}|\nabla u|^{p-2}\nabla u\nabla\phi=\int_{\Omega}a(x)u^{q-1}\phi,
\]
for all $\phi\in X$, where $X=W_{0}^{1,p}(\Omega)$ in the Dirichlet case, and
$X=W^{1,p}(\Omega)$ in the Neumann case. Since $a$ is bounded, by standard
regularity for quasilinear elliptic equations \cite{db,L}, we know that $u\in
C^{1,\alpha}(\overline{\Omega})$ for some $\alpha\in(0,1)$. If, in addition,
$u>0$ in $\Omega$, then we call it a \textit{positive solution} of $(P_{q})$.

One of the main features of $(P_{q})$ under the current conditions on $a$ and
$q$ is the possible existence of nontrivial \textit{dead core} solutions, i.e.
solutions vanishing in open subsets of $\Omega$ (see \cite{BPT,KRQU16} for
examples when $p=2$). On the
other hand, this phenomenon does not occur when $a\geq0$ or $q\geq p$, as in
this case the strong maximum principle \cite{Va} yields that any nontrivial
solution of $(P_{q})$ is positive and, by the Hopf lemma, it satisfies
$\partial_{\nu}u(x)<0$ for every $x\in\partial\Omega$ such that $u(x)=0$.

The existence of a nontrivial solution of $(P_{q})$ is not difficult to
establish, either by variational arguments or by the sub-supersolutions
method, while the existence of positive solutions is far more involved, even
for $p=2$. We shall focus here on a variational approach. Thanks to the
homogeneity in both sides of the equation, $(P_{q})$ can be tackled by several
minimization techniques (not only for $1<q<p$, but also for $p<q<p^{\ast}$,
where $p^{\ast}$ is the critical Sobolev exponent).
For $1<q<p$ we shall exploit two of them, namely, global and constrained
minimization, which we describe in the sequel. Let $I_{q}$ be the functional
given by
\[
I_{q}(u):=\int_{\Omega}\left(  \frac{1}{p}|\nabla u|^{p}-\frac{1}%
{q}a(x)|u|^{q}\right)  ,
\]
for $u\in X$. One may easily check that $I_{q}$ has a minimizer $U\geq0$,
which solves then $(P_{q})$, and satisfies $U>0$ in $\Omega_{a}^{+}$, where
\[
\Omega_{a}^{+}:=\{x\in\Omega:a(x)>0\}.
\]
We call such $U$ a \textit{ground state} (or least energy) solution of
$(P_{q})$. Alternatively, one can find a nonnegative minimizer of
$\int_{\Omega}|\nabla u|^{p}$ over the $C^{1}$ manifold
\[
\mathcal{S}_{a}:=\left\{  u\in X:\int_{\Omega}a(x)|u|^{q}=1\right\}  .
\]
By the Lagrange multipliers rule, this minimizer solves $(P_{q})$, up to some
rescaling constant. We shall see in Lemma \ref{l1} that these minimization
procedures are equivalent, i.e. they provide the same solutions. Furthermore,
these solutions turn out to be only one, cf. \cite[Theorem 1.1]{KLP}. On the
other hand, an application of a generalized Picone's inequality
\cite[Proposition 2.9]{BF} shows that this solution is the only possible
\textit{positive} solution of $(P_{q})$. More precisely:

\begin{theorem}
\label{t} For any $1<q<p$ there exists exactly one ground state solution
$U_{q}$, which is the only solution of $(P_{q})$ such that $U_{q}>0$ in
$\Omega_{a}^{+}$. In particular, $(P_{q})$ has at most one positive solution.
\end{theorem}

Uniqueness results for positive solutions of sublinear type problems have a
long history, since the well-known paper by Brezis and Oswald \cite{BO}, which
applies in the Dirichlet case to $(P_{q})$ if $p=2$ and $a>0$ in $\Omega$.
This result was extended to $p\neq2$ by D\'{\i}az and Saa \cite{DS} (see also
\cite{belloni, drabek} and its references). The indefinite case, i.e. with $a$
sign-changing, has received less attention. To the best of our knowledge, this
case has been considered only for $p=2$. Assuming that $\Omega_{a}^{+}$ is
smooth and has a finite number of connected components, Bandle et al proved
Theorem \ref{t} for the Dirichlet problem \cite[Theorem 2.3]{bandle}, and the
uniqueness of a solution positive on $\overline{\Omega_{a}^{+}}$ for the
Neumann problem \cite[Lemma 3.1]{BPT}. Still for $p=2$, Delgado and Suarez
\cite[Theorem 2.1]{DS} extended the uniqueness results for the Dirichlet case
without any assumptions on $\Omega_{a}^{+}$. Let us note that
\cite{bandle,BPT,DS} deal with more general nonlinearities (not necessarily
powerlike), and their uniqueness results are based on a change of variables
and the strong maximum principle.

The uniqueness of positive solution for $(P_{q})$ derived from Theorem \ref{t}
confirms a striking difference (known when $p=2$) with the case $p<q<p^{\ast}$
, where a high number of positive solutions may be obtained in accordance with
the number of connected components of $\Omega_{a}^{+}$, cf. \cite{BGH}. We are
not aware of an extension of this multiplicity result to $p\neq2$. Note also
that the condition '$u>0$ in $\Omega_{a}^{+}$' is sharp in the uniqueness
statement, for $(P_{q})$ may have multiple solutions that are positive in some
connected component of $\Omega_{a}^{+}$, as shown in \cite{bandle} for $p=2$.
We also extend this uniqueness feature to solutions that are positive in a
prescribed number of connected components of $\Omega_{a}^{+}$ and vanish in
the remaining ones (see Proposition \ref{ru}).

Let us consider now minimizers of $I_{q}$ in general (not only nonnegative
ones). When $\Omega_{a}^{+}$ is connected, every such minimizer has constant
sign, cf. \cite[Theorem 1.2]{KLP}, so that $\pm U_{q}$ are the only minimizers
of $I_{q}$. However, when $\Omega_{a}^{+}$ is disconnected this is no longer
true. An example of a sign-changing minimizer of $I_{q}$ is given in
\cite[Example 6.3]{KLP} for $q=1$ and $p=2$ (see also Remark \ref{rm} below
for an example with $1<q<p=2$). On the other hand, minimizers of $I_{q}$ have
constant sign whenever $U_{q}>0$ in $\Omega$. Indeed, since $|U|$ minimizes
$I_{q}$ whenever $U$ does, by Theorem \ref{t}, we have $|U|\equiv U_{q}$. Thus
$U$ does not change sign if $U_{q}>0$. This occurs when $q=p$ (in which case
$U_{q}$ has to be understood as a positive eigenfunction of $(P_{q})$), thanks
to the strong maximum principle. By some sort of continuity, this property
holds also for $q$ close to $p$:

\begin{theorem}
\label{ts} Given $a \in C(\overline{\Omega})$ there exists $q_{0}=q_{0}(a)
\in(1,p)$ such that any minimizer of $I_{q}$ has constant sign and $U_{q}$ is
the only positive solution of $(P_{q})$ for $q \in(q_{0},p)$.
\end{theorem}

We point out that the first assertion in Theorem \ref{ts} seems to be new even
for $p=2$. It can be considered as an extension of the fact that the first
positive eigenvalue of
\[
\left\{
\begin{array}
[c]{lll}%
-\Delta_{p}u=\lambda a(x)|u|^{p-2}u & \mathrm{in} & \Omega,\\
\mathbf{B}u=0 & \mathrm{on} & \partial\Omega,
\end{array}
\right.
\]
is \textit{principal}, i.e., its eigenfunctions have constant sign. As for the
second assertion, it extends (together with Theorem \ref{t}) to $p\neq2$ some
of the results in \cite[Theorem 1.2]{KRQU3}. To the best of our knowledge,
apart from \cite{bams} where the one-dimensional Dirichlet problem is
considered, this is the first result (in the sublinear and indefinite case) on
the existence of a \textit{positive} solution of $(P_{q})$ with $p\neq2$, for
both Dirichlet and Neumann boundary conditions. For the case $p=2$ we refer to
\cite{KRQU16,KRQU3} and references therein.

\begin{remark}
\strut

\begin{enumerate}
\item For the sake of simplicity, we have assumed that $a\in C(\overline
{\Omega})$. However, our results hold also if $a\in L^{\infty}(\Omega)$. In
this case we set $\Omega_{a}^{+}$ as the largest open set where $a>0$ a.e.

\item Since the proof of Theorem \ref{t} does not rely on the strong maximum
principle, it holds more generally if $\Omega$ is a bounded domain (not
necessarily smooth). In this way, we also improve (in the powerlike case) the
uniqueness results in \cite{DS}, where $\Omega$ is assumed to be smooth, and
\cite{bandle,BPT}, where $\Omega_{a}^{+}$ is required to be smooth and to have
finitely many connected components.

\item We believe that for $q$ close to $p$ the ground state solution $U_{q}$
is the \textit{unique nontrivial }solution of $(P_{q})$. This result is known
for $p=2$, assuming that $\Omega_{a}^{+}$ has finitely many connected
components, cf. \cite{KRQU16}.
\end{enumerate}
\end{remark}

The proofs of Theorems \ref{t} and \ref{ts} are divided into several
Propositions and Lemmae, stated in the next section.

\subsection*{Notation}

Throughout this paper, we use the following notation:

\begin{itemize}

\item $\Omega_{f}^{+}:=\{x\in\Omega:f(x)>0\}$ for $f\in C(\overline{\Omega})$.

\item Given $u$ such that $\int_{\Omega}a(x)|u|^{q}>0$ we denote by $\tilde
{u}$ the projection of $u$ over $\mathcal{S}_{a}$, i.e. $\tilde{u}:=\left(
\int_{\Omega}a(x)|u|^{q}\right)  ^{-\frac{1}{q}}u$.

\item Given $r>1$, we denote by $\Vert\cdot\Vert_{r}$ the usual norm in
$L^{r}(\Omega)$ and by $\Vert\cdot\Vert$ the usual norm in $X$, i.e. $\Vert
u\Vert=\Vert\nabla u\Vert_{p}$ if $X=W_{0}^{1,p}(\Omega)$ and $\Vert
u\Vert=\Vert\nabla u\Vert_{p}+\Vert u\Vert_{p}$ if $X=W^{1,p}(\Omega)$.

\item If $A\subset\mathbb{R}^{N}$ then we denote by $\mathds{1}_{A}$ the
characteristic function of $A$.
\end{itemize}

\section{Proofs}

We set
\[
M:=\inf_{u \in X} I_{q}(u) \quad\text{and} \quad m:=\inf_{v \in\mathcal{S}%
_{a}} \int_{\Omega}|\nabla v|^{p}.
\]

Let us show that these infima provide the same solutions of $(P_{q})$, and
these ones are positive in $\Omega_{a}^{+}$:

\begin{lemma}
\label{l1} \strut

\begin{enumerate}
\item There exists $U\in X$ such that $U \geq0$ and $\displaystyle I_{q}%
(U)=M<0$.

\item There exists $V \in\mathcal{S}_{a}$ such that $V \geq0$ and
$\int_{\Omega}|\nabla V|^{p}=m>0$.

\item If $\displaystyle I_{q}(U)=M$ then $\int_{\Omega}a(x)|U|^{q}>0$ and
$\int_{\Omega}|\nabla\tilde{U}|^{p}=m$.

\item If $\int_{\Omega}|\nabla V|^{p}=m$ and $V \in\mathcal{S}_{a}$ then
$I_{q}(CV)=M$ for some $C>0$.

\item If $\displaystyle I_{q}(U)=M$ and $U \geq0$ then $U>0$ in $\Omega
_{a}^{+}$.

\item If $\int_{\Omega}|\nabla V|^{p}=m$, $V \in\mathcal{S}_{a}$, and $V
\geq0$ then $V>0$ in $\Omega_{a}^{+}$.
\end{enumerate}
\end{lemma}

\begin{proof}
\strut
\begin{enumerate}
\item The proof follows by standard compactness arguments. Let us first show
that $M<0$. Indeed, let $u\in X$ be such that $\int_{\Omega}a(x)|u|^{q}>0$.
Then, for $t>0$ small enough, we have
\[
I_{q}(tu)=\frac{t^{p}}{p}\int_{\Omega}|\nabla u|^{p}-\frac{t^{q}}{q}%
\int_{\Omega}a(x)|u|^{q}<0,
\]
since $q<p$. If $X=W_{0}^{1,p}(\Omega)$ then, by Sobolev and Holder
inequalities, we find some constant $C>0$ such that
\[
I_{q}(u)\geq\frac{1}{p}\Vert u\Vert^{p}-C\Vert u\Vert^{q}\quad\forall u\in X,
\]
i.e. $I$ is coercive. Now, if $X=W^{1,p}(\Omega)$ then we claim that there
exists a constant $C_{1}>0$ such that
$\int_{\Omega}|\nabla u|^{p}\geq C_{1}\Vert u\Vert^{p}$ for every $u\in X$
such that $\int_{\Omega}a(x)|u|^{q}\geq0$. Indeed, otherwise there exists a
sequence $(u_{n})\subset X$ such that
\[
\int_{\Omega}a(x)|u_{n}|^{q}\geq0,\quad\int_{\Omega}|\nabla u_{n}%
|^{p}\rightarrow0,\quad\text{and}\quad\Vert u_{n}\Vert=1
\]
for every $n$. Then, up to a subsequence, we have $u_{n}\rightarrow k$ in $X$,
for some constant $k\neq0$. Since $\int_{\Omega}a(x)|u_{n}|^{q}\geq0$ it
follows that $\int_{\Omega}a\geq0$, a contradiction. Thus the claim is proved
and it implies that
\[
I_{q}(u)\geq%
\begin{cases}
C_{1}\Vert u\Vert^{p}-C\Vert u\Vert^{q} & \mbox{ if }\int_{\Omega}%
a(x)|u|^{q}\geq0,\\
0 & \mbox{ if }\int_{\Omega}a(x)|u|^{q}<0,
\end{cases}
\]
so that $I$ is bounded from below in $X$ in both cases. Therefore, since
$I_{q}$ is weakly lower semi-continuous we deduce that $I_{q}(U)=M<0$ for some
$U\in X$, which can be chosen nonnegative, since $I_{q}(u)=I_{q}(|u|)$.
\item Since $v\mapsto\int_{\Omega}|\nabla v|^{p}$ is weakly lower
semi-continuous and $\mathcal{S}_{a}$ is weakly closed in $X$, we see that
there exists $V\in\mathcal{S}_{a}$ such that $\displaystyle\int_{\Omega
}|\nabla V|^{p}=m$. Moreover since $\int_{\Omega}|\nabla v|^{p}=\int_{\Omega
}|\nabla|v||^{p}$, we can take $V\geq0$. Finally, if $X=W^{1,p}(\Omega)$ then
$V$ is not a constant, in view of the condition $\int_{\Omega}a<0$.
\item Let $U$ be such that $I_{q}(U)=M$. Since $M<0$ we have that
$\int_{\Omega}a(x)|U|^{q}>0$. Let $V\in\mathcal{S}_{a}$ be
such that $\displaystyle\int_{\Omega}|\nabla V|^{p}=m$. Then
\[
I_{q}(U)\leq I_{q}(tV)=\frac{t^{p}}{p}\int_{\Omega}|\nabla V|^{p}-\frac{t^{q}%
}{q}\int_{\Omega}a(x)|V|^{q}=\frac{t^{p}}{p}m-\frac{t^{q}}{q}%
\]
for any $t\in\mathbb{R}$. We choose $t=\left(  \int_{\Omega}a(x)|U|^{q}%
\right)  ^{\frac{1}{q}}$, so that
\[
\frac{1}{p}\int_{\Omega}|\nabla U|^{p}-\frac{1}{q}\int_{\Omega}a(x)|U|^{q}%
=I_{q}(U)\leq\frac{1}{p}\left(  \int_{\Omega}a(x)|U|^{q}\right)  ^{\frac{p}%
{q}}m-\frac{1}{q}\int_{\Omega}a(x)|U|^{q},
\]
i.e.
\[
\int_{\Omega}|\nabla U|^{p}\leq\left(  \int_{\Omega}a(x)|U|^{q}\right)
^{\frac{p}{q}}m.
\]
Thus $\int_{\Omega}|\nabla\tilde{U}%
|^{p}\leq m$, which yields the desired conclusion.
\item We use a similar trick. Let $U$ be as in the first item. Then, by the
previous item,
\[
I_{q}(tV)=\frac{t^{p}}{p}m-\frac{t^{q}}{q}=\frac{t^{p}}{p}\int_{\Omega}%
|\nabla\tilde{U}|^{p}-\frac{t^{q}}{q}=\frac{t^{p}}{p}\frac{\int_{\Omega
}|\nabla U|^{p}}{\left(  \int_{\Omega}a(x)|U|^{q}\right)  ^{\frac{p}{q}}%
}-\frac{t^{q}}{q},
\]
so that, taking $t=\left(  \int_{\Omega}a(x)|U|^{q}\right)  ^{\frac{1}{q}}$,
we find that
\[
I_{q}(tV)=\frac{1}{p}\int_{\Omega}|\nabla U|^{p}-\frac{1}{q}\int_{\Omega
}a(x)|U|^{q}=M,
\]
which concludes the proof.
\item Let $U\geq0$ be such that $I_{q}(U)=M$. If $U(x_{0})=0$ for some
$x_{0}\in\Omega_{a}^{+}$ then, by the strong maximum principle, $U\equiv0$ in
a ball $B\subset\Omega_{a}^{+}$. We choose then $\phi\in C_{0}^{\infty}(B)$
with $\phi\geq0,\not \equiv 0$. Then, for $t>0$ small enough, we have
$I_{q}(t\phi)<0$, so that
\[
I_{q}(U+t\phi)=I_{q}(U)+I_{q}(t\phi)<I_{q}(U)=M,
\]
and we obtain a contradiction. Thus $U>0$ in $\Omega_{a}^{+}$.
\item It follows from (4) and (5).
\end{enumerate}
\end{proof}

Let us prove now that $m$ is achieved by exactly one nonnegative minimizer,
which we denote by $V_{q}$ from now on. This result was proved in \cite{KLP}
in a more general setting, but we include the proof here for completeness. It
relies on the following inequality, which is a particular case of
\cite[Proposition 6.1]{KLP}:

\begin{lemma}
\label{in} Let $q \in[1,p]$ and $\alpha_{1},\alpha_{2} \in[0,1]$ with
$\alpha_{1}^{q}+\alpha_{2}^{q}=1$. Then, for any $\eta_{1},\eta_{2}
\in\mathbb{R}^{N}$, we have
\[
\left|  \alpha_{1}^{q-1}\eta_{1}+\alpha_{2}^{q-1}\eta_{2} \right|  ^{p}
\leq2^{\frac{p}{q}-1}(|\eta_{1}|^{p}+|\eta_{2}|^{p}),
\]
with strict inequality if $\alpha_{1}\neq\alpha_{2}$ and $|\eta_{1}|+|\eta
_{2}|\neq0$.
\end{lemma}

\begin{proposition}
\label{um} There exists exactly one $V_{q} \in\mathcal{S}_{a}$ such that
$V_{q} \geq0$ and $\int_{\Omega}|\nabla V_{q}|^{p}=m$.
\end{proposition}

\begin{proof}
Assume that $V_{1},V_{2}\geq0$ satisfy
\[
\int_{\Omega}|\nabla V_{1}|^{p}=\int_{\Omega}|\nabla V_{2}|^{p}=m\quad
\text{and}\quad\int_{\Omega}a(x)V_{1}^{q}=\int_{\Omega}a(x)V_{2}^{q}=1.
\]
We set $W:=\left(  \frac{V_{1}^{q}+V_{2}^{q}}{2}\right)  ^{\frac{1}{q}}$,
so that $\int_{\Omega}a(x)W^{q}=1$, and
\begin{align*}
\nabla W  &  =\frac{1}{2}\left(  \frac{V_{1}^{q}+V_{2}^{q}}{2}\right)
^{\frac{1-q}{q}}\left(  V_{1}^{q-1}\nabla V_{1}+V_{2}^{q-1}\nabla
V_{2}\right)  \mathds{1}_{\Omega_{V_{1}}^{+}\cup\Omega_{V_{2}}^{+}}\\
&  =2^{-\frac{1}{q}}\left[  \left(  \frac{V_{1}^{q}}{V_{1}^{q}+V_{2}^{q}%
}\right)  ^{\frac{q-1}{q}}\nabla V_{1}+\left(  \frac{V_{2}^{q}}{V_{1}%
^{q}+V_{2}^{q}}\right)  ^{\frac{q-1}{q}}\nabla V_{2}\right]
\mathds{1}_{\Omega_{V_{1}}^{+}\cup\Omega_{V_{2}}^{+}}.
\end{align*}
We apply Lemma \ref{in} with $\alpha_{1}=\left(  \frac{V_{1}^{q}}{V_{1}%
^{q}+V_{2}^{q}}\right)  ^{\frac{1}{q}}$, $\alpha_{2}=\left(  \frac{V_{2}^{q}%
}{V_{1}^{q}+V_{2}^{q}}\right)  ^{\frac{1}{q}}$, $\eta_{1}=\nabla V_{1}$, and
$\eta_{2}=\nabla V_{2}$. Thus
\[
\left\vert \left(  \frac{V_{1}^{q}}{V_{1}^{q}+V_{2}^{q}}\right)  ^{\frac
{q-1}{q}}\nabla V_{1}+\left(  \frac{V_{2}^{q}}{V_{1}^{q}+V_{2}^{q}}\right)
^{\frac{q-1}{q}}\nabla V_{2}\right\vert ^{p}\leq2^{\frac{p}{q}-1}\left(
|\nabla V_{1}|^{p}+|\nabla V_{2}|^{p}\right)
\]
in $\Omega_{V_{1}}^{+}\cup\Omega_{V_{2}}^{+}$, with strict inequality in the
set
\[
E:=\{x\in\Omega_{V_{1}}^{+}\cup\Omega_{V_{2}}^{+}:V_{1}(x)\neq V_{2}%
(x),|\nabla V_{1}(x)|+|\nabla V_{2}(x)|\neq0\}.
\]
It follows that
\[
\int_{\Omega}|\nabla W|^{p}\leq2^{-\frac{p}{q}}\int_{\Omega_{V_{1}}^{+}%
\cup\Omega_{V_{2}}^{+}}2^{\frac{p}{q}-1}\left(  |\nabla V_{1}|^{p}+|\nabla
V_{2}|^{p}\right)  \leq m.
\]
Thus $\int_{\Omega}|\nabla W|^{p}=m$ and $|E|=0$, so that for almost every
$x\in\Omega$ we have
\[
V_{1}(x)=V_{2}(x)\quad\text{or}\quad\nabla V_{1}(x)=\nabla V_{2}(x)=0
\]
In particular, $\nabla V_{1}=\nabla V_{2}$ a.e. in $\Omega$, so $V_{1}\equiv
V_{2}+C$, for some constant $C$. If $C\neq0$ then, from the alternative above,
we have $\nabla V_{1}=\nabla V_{2}=0$ a.e. in $\Omega$, which is impossible.
Therefore $V_{1}\equiv V_{2}$, and the proof is complete.
\end{proof}

From Lemma \ref{l1}-(3) we deduce that $I_{q}$ has a unique nonnegative
minimizer, and we denote it by $U_{q}$ from now on.

\begin{corollary}
\label{un} There exists exactly one $U_{q} \in X$ such that $U_{q}\geq0$ and
$I_{q}(U_{q})=M$.
\end{corollary}

\begin{proof}
If $U_1,U_2 \in X$ satisfy $U_1,U_2 \geq 0$ and $I_q(U_1)=I_q(U_2)=M$ then, by  Lemma \ref{l1}-(3) and Proposition \ref{um}, we have $\tilde{U}_1 \equiv \tilde{U}_2 \equiv V_q$. Thus $U_1=CU_2$, for some $C>0$. But since $U_1$ and $U_2$ solve $(P_q)$, we infer that $C=1$.
\end{proof}

\begin{remark}
\label{rm} When $\Omega_{a}^{+}$ is connected, every minimizer of $I_{q}$ has
a sign, cf. \cite[Theorem 1.2]{KLP}. However, when $\Omega_{a}^{+}$ is
disconnected $I_{q}$ may have a sign-changing minimizer, cf. \cite[Example
6.3]{KLP} for $q=1$ and $p=2$. More generally, for $1<q<p=2$, this situation
occurs, for instance, if $\Omega=(b,c)$, and $\Omega_{a}^{+}=(b,b+\delta
)\cup(c-\delta,c)$, for some $\delta>0$. If $a$ is sufficiently negative in
$(b+\delta,c-\delta)$ then any solution of $(P_{q})$ vanishes in a subinterval
of $(b+\delta,c-\delta)$, cf. \cite[Theorem 3.2]{ans}. Thus $U_{q}$ has two
positive bumps, so that changing the sign of one of these bumps one gets a
sign-changing minimizer of $I_{q}$.
\end{remark}

The next step is to show that $U_{q}$ is the only solution of $(P_{q})$
satisfying $U_{q}>0$ in $\Omega_{a}^{+}$. This result, which has been proved
in \cite[Theorem 5.1]{BF} for $a \equiv1$, is based on the following
generalized Picone's identity (or inequality). We also include a (simpler)
proof here, since \cite[Proposition 2.9]{BF} deals with a more general
differential operator. Note that when $q=p$ we obtain the usual Picone's
identity, which has been used to prove the simplicity of the first
$p$-Laplacian eigenvalue (among other results), cf. \cite{AH}.

\begin{lemma}
[Generalized Picone's identity]\label{lp} Let $q \in[1,p]$ and $u,v \in
W^{1,p}(\Omega)$ with $u>0$ and $v\geq0$. Then
\[
|\nabla u|^{p-2} \nabla u \nabla\left(  \frac{v^{q}}{u^{q-1}}\right)
\leq|\nabla u|^{p-q} |\nabla v|^{q}.
\]

\end{lemma}

\begin{proof}
Note that
$$|\nabla u|^{p-2} \nabla u \nabla \left(\frac{v^q}{u^{q-1}}\right)=q\left(\frac{v}{u}\right)^{q-1} |\nabla u|^{p-2} \nabla u \nabla v -(q-1)\left(\frac{v}{u}\right)^q |\nabla u|^p$$
We apply Young's inequality $ab \leq \frac{a^r}{r}+\frac{b^{r'}}{r'}$ with $a=\left(\frac{v}{u}\right)^{q-1} |\nabla u|^{\frac{p(q-1)}{q}}$, $b=|\nabla u|^{\frac{p-q}{q}}|\nabla v|$, and $r=\frac{q}{q-1}$, so that $r'=q$. Thus
$$\left(\frac{v}{u}\right)^{q-1} |\nabla u|^{p-1}  |\nabla v|=ab \leq \frac{q-1}{q}\left(\frac{v}{u}\right)^q |\nabla u|^p+\frac{1}{q}|\nabla u|^{p-q}  |\nabla v|^q,$$
which yields the desired conclusion.
\end{proof}

\begin{proposition}
\label{us} If $u$ is a solution of $(P_{q})$ such that $u>0$ in $\Omega
_{a}^{+}$ then $u \equiv U_{q}$.
\end{proposition}

\begin{proof}
Let $\epsilon>0$.
We take $\frac{V_q^{q}}{(u+\epsilon)^{q-1}}$ as test function in $(P_{q})$ and
apply Lemma \ref{lp} (with $u+\epsilon$ instead of $u$) to obtain
\[
\int_{\Omega}a(x)u^{q-1}\frac{V_q^{q}}{(u+\epsilon)^{q-1}}=\int_{\Omega}|\nabla
u|^{p-2}\nabla u\nabla\left(  \frac{V_q^{q}}{(u+\epsilon)^{q-1}}\right)
\leq\int_{\Omega}|\nabla u|^{p-q}|\nabla V_q|^{q}.
\]
Now, by Holder's inequality we find that
\[
\int_{\Omega}|\nabla u|^{p-q}|\nabla V_q|^{q}\leq\left(  \int_{\Omega}|\nabla
u|^{p}\right)  ^{\frac{p-q}{p}}\left(  \int_{\Omega}|\nabla V_q|^{p}\right)
^{\frac{q}{p}}=m^{\frac{q}{p}}\left(  \int_{\Omega}|\nabla u|^{p}\right)
^{\frac{p-q}{p}}.
\]
Note that $\frac{u}{u+\epsilon}\rightarrow\mathds{1}_{\Omega_{u}^{+}}$ as
$\epsilon\rightarrow0$. Thus, by Lebesgue's dominated convergence theorem and
the above inequalities, we have
\[
\int_{\Omega_{u}^{+}}a(x)|V_q|^{q}=\lim_{\epsilon\rightarrow0}\int_{\Omega
}a(x)|V_q|^{q}\left(  \frac{u}{u+\epsilon}\right)  ^{q-1}\leq m^{\frac{q}{p}%
}\left(  \int_{\Omega}|\nabla u|^{p}\right)  ^{\frac{p-q}{p}}.
\]
In addition, since $u>0$ in $\Omega_{a}^{+}$, we have $a\leq0$ in
$\Omega\setminus\Omega_{u}^{+}$, which implies that
\[
\int_{\Omega_{u}^{+}}a(x)|V_q|^{q}=1-\int_{\Omega\setminus\Omega_{u}^{+}%
}a(x)|V_q|^{q}\geq1,
\]
and therefore
\[
\left(  \int_{\Omega}|\nabla u|^{p}\right)  ^{\frac{q-p}{q}}\leq m.
\]
Now, since $u$ solves $(P_{q})$, we have $\int_{\Omega}|\nabla u|^{p}%
=\int_{\Omega}a(x)u^{q}$, so the latter inequality yields
\[
\int_{\Omega}|\nabla\tilde{u}|^{p}=\frac{\int_{\Omega}|\nabla u|^{p}}{\left(
\int_{\Omega}a(x)u^{q}\right)  ^{\frac{p}{q}}}\leq m,
\]
i.e. $\tilde{u}\equiv V_{q}$. By Lemma \ref{l1}-(4) and Corollary \ref{un}, we
conclude that $u\equiv U_{q}$.
\end{proof}

Next we prove a generalization of the uniqueness assertion in Proposition
\ref{us}. This result extends \cite[Theorem 2.1]{bandle} to $p\neq2$, without
requiring any smoothness condition on $\Omega_{a}^{+}$, nor the finiteness of
$\mathcal{J}$.

\begin{proposition}
\label{ru}Let $\{\Omega_{i}:i\in\mathcal{I}\}$ be the connected components of
$\Omega_{a}^{+}$, and $\mathcal{J}\subset\mathcal{I}$. Then $(P_{q})$ has at
most one solution such that $u>0$ in $\displaystyle\bigcup_{i\in\mathcal{J}%
}\Omega_{i}$ and $u\equiv0$ in $\displaystyle\bigcup_{i\in\mathcal{I}%
\setminus\mathcal{J}}\Omega_{i}$.
\end{proposition}

\begin{proof} Set $ m_{j}:=\inf\left\{  \int_{\Omega}|\nabla
v|^{p}:v\in\mathcal{S}_{a}\text{ and }v\equiv0\mbox{ in }\bigcup
_{i\in\mathcal{I}\setminus\mathcal{J}}\Omega_{i}\right\}  .$ Arguing as in
Proposition \ref{um}, we can show that $m_{j}$ is achieved by a unique
$V_{j}\geq0$. Repeating the proof of Proposition \ref{us}, with $V_{j}$
instead of $V$, we obtain
\[
\int_{\Omega_{u}^{+}}a(x)|V_{j}|^{q}=\lim_{\epsilon\rightarrow0}\int_{\Omega
}a(x)|V_{j}|^{q}\left(  \frac{u}{u+\epsilon}\right)  ^{q-1}\leq m_{j}%
^{\frac{q}{p}}\left(  \int_{\Omega}|\nabla u|^{p}\right)  ^{\frac{p-q}{p}}.
\]
Now, since $V_{j}=0$ in $\displaystyle\bigcup_{i\in\mathcal{I}\setminus
\mathcal{J}}\Omega_{i}$, we have $a(x)V_{j}^{q}\leq0$ in $\Omega
\setminus\Omega_{u}^{+}$, so that
\[
1=\int_{\Omega}a(x)|V_{j}|^{q}=\int_{\Omega_{u}^{+}}a(x)|V_{j}|^{q}%
+\int_{\Omega\setminus\Omega_{u}^{+}}a(x)|V_{j}|^{q}\leq\int_{\Omega_{u}^{+}%
}a(x)|V_{j}|^{q}.
\]
The rest of the argument yields that $\int_{\Omega}|\nabla\tilde{u}|^{p}\leq
m_{j}$, i.e. $\tilde{u}\equiv V_{j}$.
\end{proof}

The existence of solutions as the ones in the aforementioned proposition is a
more delicate issue that requires some conditions on $a$ and $q$ allowing dead
cores formation in $(P_{q})$. When $p=2$, we know that these solutions
\textit{do not} exist for $q$ close enough to $p$, cf. \cite{KRQU16}%
.\smallskip

We prove now that minimizers of $I_{q}$ do not change sign when $q$ is close
to $p$. Recall that 
\[
\lambda_{1}(a)=\min\left\{  \int_{\Omega}|\nabla v|^{p}:v\in X,\int_{\Omega
}a(x)|v|^{p}=1\right\}
\]
is the first positive eigenvalue of
\[
\left\{
\begin{array}
[c]{lll}%
-\Delta_{p}u=\lambda a(x)|u|^{p-2}u & \mathrm{in} & \Omega,\\
\mathbf{B}u=0 & \mathrm{on} & \partial\Omega.
\end{array}
\right.
\]
We denote by $\phi_{1}(a)$ a positive eigenfunction associated to $\lambda
_{1}(a)$. By the strong maximum principle, we have $\phi_{1}(a)\in
\mathcal{P}^{\circ}$, where
\[
\mathcal{P}^{\circ}:=\left\{
\begin{array}
[c]{ll}%
\left\{  u\in C_{0}^{1}(\overline{\Omega}):u>0\ \mbox{in $\Omega$},\ \partial
_{\nu}u<0\ \mbox{on $\partial \Omega$}\right\}   &
\mbox{if ${\bf B }u= u$},\medskip\\
\left\{  u\in C^{1}(\overline{\Omega}%
):u>0\ \mbox{on $\overline{\Omega}$}\right\}   &
\mbox{if ${\bf B }u=\partial_\nu u$}.
\end{array}
\right.
\]

\begin{proposition}
\label{cs} There exists $q_{0}=q_{0}(a)\in(1,p)$ such that $U_{q}%
\in\mathcal{P}^{\circ}$ for $q\in(q_{0},p)$. In particular, any minimizer of
$I_{q}$ has constant sign for $q\in(q_{0},p)$.
\end{proposition}

\begin{proof}
Since $\lambda_{1}(ca)=c^{-1}\lambda_{1}(a)$ for any $c>0$,
and $u$ solves $(P_{q})$ if, and only if, $c^{\frac{1}{p-q}}u$ solves $(P_{q})$
with $a$ replaced by $ca$, we can assume without loss of generality that
$\lambda_{1}(a)<1$. Assume by contradiction that there exists a sequence
$q_{n}\rightarrow p^{-}$ with $u_{n}\not \in \mathcal{P}^{\circ}$, where
$u_{n}:=U_{q_{n}}$. First we assume that $(u_{n})$ is bounded in $X$, so that,
up to a subsequence, $u_{n}\rightharpoonup u_{0}$ in $X$, $u_{n}\rightarrow
u_{0}$ in $L^{t}(\Omega)$ with $t\in(1,p^{\ast})$, and $u_{n}\rightarrow
u_{0}$ \textit{a.e.} in $\Omega$, for some $u_{0}\in X$. It follows that
$u_{0}\geq0$ and
\begin{align*}
\frac{1}{p}\int_{\Omega}\left(  |\nabla u_{0}|^{p}-a(x)|u_{0}|^{p}\right)   &
\leq\liminf I_{q_{n}}(u_{n})\leq\liminf I_{q_{n}}(\phi_{1})\\
&  =\frac{1}{p}\int_{\Omega}\left(  |\nabla\phi_{1}(a)|^{p}-a(x)\phi
_{1}(a)^{p}\right) \\
&  <\frac{1}{p}\int_{\Omega}\left(  |\nabla\phi_{1}(a)|^{p}-\lambda
_{1}(a)a(x)\phi_{1}(a)^{p}\right)  =0,
\end{align*}
which shows that $u_{0}\not \equiv 0$. Moreover, one can easily see that
$u_{n}\rightarrow u_{0}$ in $X$ and $u_{0}$ solves $-\Delta_{p}u_{0}%
=a(x)u_{0}^{p-1}$ in $\Omega$. Since $\int_{\Omega}au_{n}^{q_{n}}>0$, we have
$\int_{\Omega}au_{0}^{p}>0$, and consequently $\lambda_{1}(a)=1$, a
contradiction. Thus $(u_{n})$ is unbounded in $X$. We can assume
that
\[
\left\Vert u_{n}\right\Vert \rightarrow\infty,\quad v_{n}:=\frac{u_{n}%
}{\left\Vert u_{n}\right\Vert }\rightharpoonup v_{0}\text{ in }X,\quad
v_{n}\rightarrow v_{0}\text{ in }L^{t}(\Omega)\text{ }\mbox{with }t\in
(1,p^{\ast}),
\]
for some $v_{0}\in X$. Note that $v_{n}$ satisfies
\begin{equation}
-\Delta_{p}v_{n}=a(x)\frac{v_{n}^{q_{n}-1}}{\left\Vert u_{n}\right\Vert
^{p-q_{n}}},\quad v_{n}\geq0,\quad v_{n}\in X. \label{ev}%
\end{equation}
Since $\left\Vert u_{n}\right\Vert \geq1$ for $n$ large enough, we have either
$\left\Vert u_{n}\right\Vert ^{p-q_{n}}\rightarrow\infty$ or $\left\Vert
u_{n}\right\Vert ^{p-q_{n}}$ is bounded. In the first case, from \eqref{ev} we
have
\[
\int_{\Omega}|\nabla v_{n}|^{p}=\frac{\int_{\Omega}a(x)v_{n}^{q_{n}}%
}{\left\Vert u_{n}\right\Vert ^{p-q_{n}}}\rightarrow0,
\]
which is a contradiction. Now, if $\left\Vert u_{n}\right\Vert ^{p-q_{n}}$ is
bounded then we can assume that $\left\Vert u_{n}\right\Vert ^{p-q_{n}%
}\rightarrow d\geq1$. From \eqref{ev}, we obtain
\[
\int_{\Omega}|\nabla v_{0}|^{p-2}\nabla v_{0}\nabla\phi=\frac{1}{d}%
\int_{\Omega}a(x)v_{0}^{p-1}\phi,\quad\forall\phi\in X,
\]
i.e.
\[
-\Delta_{p}v_{0}=\frac{1}{d}a(x)v_{0}^{p-1}\quad\text{in }\Omega,\quad
v_{0}\in X.
\]
In addition, $v_{n}\rightarrow v_{0}$ in $X$, so that $v_{0}\not \equiv 0$ and
$v_{0}\geq0$ (which implies that $\lambda_1(a)=d^{-1}$). By the strong maximum principle, we deduce that $v_{0}%
\in\mathcal{P}^{\circ}$. Finally, by elliptic regularity, we find that
$v_{n}\rightarrow v_{0}$ in $C^{1}(\overline{\Omega})$. Consequently $v_{n}%
\in\mathcal{P}^{\circ}$ for $n$ large enough, which contradicts $u_{n}%
\not \in \mathcal{P}^{\circ}$. Therefore there exists $q_0(a)=q_{0}\in(1,p)$ such
that $U_{q}\in\mathcal{P}^{\circ}$ for $q\in(q_{0},p)$, which shows in
particular that any minimizer of $I_{q}$ has constant sign for such $q$.
\end{proof}

\begin{remark}
\strut

\begin{enumerate}
\item In the Dirichlet case, Proposition \ref{cs} can be extended as follows:
given $q \in(1,p)$ and $a^{+}$ fixed, there exists $\delta>0$ such that $U_{a}
\in\mathcal{P}^{\circ}$ if $\|a^{-}\|_{\infty}<\delta$, where $U_{a}$ is the
unique nonnegative minimizer of
\[
I_{a}(u)=\int_{\Omega}\left(  \frac{1}{p}|\nabla u|^{p} - \frac{1}{q}
a(x)|u|^{q}\right)  ,
\]
defined on $W_{0}^{1,p}(\Omega)$. In particular, minimizers of $I_{a}$ have
constant sign if $\|a^{-}\|_{\infty}<\delta$. Indeed, assume that $a_{n}=a^{+}
-a_{n}^{-}$, with $a_{n}^{-} \to0$ in $C(\overline{\Omega})$, and let
$u_{n}:=U_{a_{n}}$. Then $(u_{n})$ is bounded in $W_{0}^{1,p}(\Omega)$, since
\[
\int_{\Omega}|\nabla u_{n}|^{p} \leq\int_{\Omega}a^{+}(x) |u_{n}|^{q}.
\]
One can show then that $u_{n} \to u_{0}$ in $C^{1}(\overline{\Omega})$, and
$u_{0}\geq0$ solves $-\Delta_{p} u =a^{+}(x)u^{q-1}$. Moreover $u_{0}
\not \equiv 0$ since
\[
I_{a^{+}}(u_{0})=\lim I_{a_{n}}(u_{n}) \leq\lim I_{a_{n}}(u_{+})=I_{a^{+}%
}(u_{+})<0,
\]
where $u_{+}$ is the nonnegative minimizer of $I_{a^{+}}$. Thus $u_{0}
\in\mathcal{P}^{\circ}$, which yields a contradiction.

\item The proof of Proposition \ref{cs} also shows that $U_{q}$ has the
following asymptotic behavior as $q \to p^{-}$:

\begin{itemize}
\item $\|U_{q}\|_{\infty}\to\infty$ if $\lambda_{1}(a)<1$.

\item $U_{q} \to0$ in $C^{1}(\overline{\Omega})$ if $\lambda_{1}(a)>1$.
\end{itemize}

This fact has been observed for $p=2$ in \cite{KRQU3}.
\end{enumerate}
\end{remark}



\begin{thebibliography}{99}                                                                                               %




\bibitem {AH}W. Allegretto, Y.X. Huang, \textit{A Picone's identity for the
p-Laplacian and applications.} Nonlinear Anal. \textbf{32} (1998), 819--830.





\bibitem {bandle}C. Bandle, M. Pozio, A. Tesei, \textit{The asymptotic
behavior of the solutions of degenerate parabolic equations,} Trans. Amer.
Math. Soc. \textbf{303} (1987), 487--501.\textit{\ }

\bibitem {BPT}C.\ Bandle, M.\ Pozio, A.\ Tesei, \textit{Existence and
uniqueness of solutions of nonlinear Neumann problems},
Math.\ Z.\ \textbf{199} (1988), 257--278.

\bibitem {belloni}M. Belloni, {B.\ Kawohl, }\textit{A direct uniqueness proof
for equations involving the }$p$\textit{-Laplace operator}, Manuscripta Math.
\textbf{109} (2002), 229--231.

\bibitem {BGH}D. Bonheure, J. M. Gomes, P. Habets, \textit{ Multiple positive
solutions of superlinear elliptic problems with sign-changing weight}, J.
Differential Equations \textbf{214} (2005), 36--64.

\bibitem {BF}L. Brasco, G. Franzina, \textit{ Convexity properties of
Dirichlet integrals and Picone-type inequalities}, Kodai Math. J. \textbf{37}
(2014), 769--799.

\bibitem {BO}H. Brezis, L. Oswald, \textit{ Remarks on sublinear elliptic
equations,} Nonlinear Anal. \textbf{10} (1986), 55--64.



\bibitem {DS}M. Delgado, A. Su\'{a}rez, \textit{On the uniqueness of positive
solution of an elliptic equation}, Appl. Math. Lett. \textbf{18} (2005), 1089--1093.

\bibitem {DS}J.I. D\'{\i}az, J.E. Saa, \textit{Existence et unicit\'{e}\ de
solutions positives pour certaines \'{e}quations elliptiques
quasilin\'{e}aires. (French) [Existence and uniqueness of positive solutions
of some quasilinear elliptic equations] }C. R. Acad. Sci. Paris S\'{e}r. I
Math. \textbf{305} (1987), 521--524.









\bibitem {db}E. DiBenedetto, \textit{$C^{1+\alpha}$ local regularity of weak
solutions of degenerate elliptic equations}, Nonlinear Anal. \textbf{7}
(1983), 827--850.




\bibitem {drabek}P. Dr\'{a}bek, J. Hern\'{a}ndez, \textit{Existence and
uniqueness of positive solutions for some quasilinear elliptic problems},
Nonlinear Anal. \textbf{44} (2001), 189--204.

\bibitem {ans}T. Godoy, U. Kaufmann, \textit{Existence of strictly positive
solutions for sublinear elliptic problems in bounded domains}, Adv. Nonlinear
Stud. \textbf{14} (2014), 353--359.

\bibitem {bams}U. Kaufmann, I. Medri, \textit{Strictly positive solutions for
one-dimensional nonlinear problems involving the p-Laplacian}, Bull. Aust.
Math. Soc. \textbf{89} (2014), 243--251.







\bibitem {KRQU16}{U.\ Kaufmann, H.\ Ramos Quoirin, K.\ Umezu,
\textit{Positivity results for indefinite sublinear elliptic problems via a
continuity argument}, J.\ Differential Equations \textbf{263} (2017),
4481--4502. }

\bibitem {KRQU3}U.\ Kaufmann, H.\ Ramos Quoirin, K.\ Umezu, \textit{A curve of
positive solutions for an indefinite sublinear Dirichlet problem}, Discrete
Contin. Dyn. Syst.\ \textbf{40} (2020), 617--645.

\bibitem {KLP}{B.\ Kawohl, M.\ Lucia, S.\ Prashanth, \textit{Simplicity of the
principal eigenvalue for indefinite quasilinear problems}, Adv. Differential
Equations, \textbf{12} (2007), 407--434. }

\bibitem {L}G. M. Lieberman, \textit{Boundary regularity for solutions of
degenerate elliptic equations}, Nonlinear Anal. \textbf{12} (1988),
1203--1219.








\bibitem {Va}J.L. V\'{a}zquez, A strong maximum principle for some quasilinear
elliptic equations. \textit{Appl. Math. Optim}\emph{.} \textbf{12} (1984), 191--202.
\end{thebibliography}
\end{document}